\documentclass[12pt]{article}
\usepackage{mathrsfs}
\usepackage{amsmath}
\usepackage{amsmath,amsthm,amsfonts,amssymb,url}
\newtheorem{theorem}{Theorem}
\newtheorem{conjecture}[theorem]{Conjecture}

\newtheorem{lemma}[theorem]{Lemma}
\newtheorem{problem}[theorem]{Problem}
\newtheorem{definition}[theorem]{Definition}
\newtheorem{example}[theorem]{Example}
\newtheorem{remark}[theorem]{Remark}
\begin{document}

\title{Difference between minimum light numbers of sigma-game and lit-only sigma-game
\footnote{Supported by NNSFC (No.
10871128), STCSM (No. 08QA14036) and Chinese Ministry of Education
(No. 108056).}}
\author{Xinmao Wang\footnote{Department of Mathematics, University of Science and
Technology of China, Hefei, Anhui 230026, China.}~, Yaokun Wu\footnote{Department of
Mathematics, Shanghai Jiao Tong University, Shanghai 200240, China. Corresponding author.
E-mail address: ykwu@sjtu.edu.cn}}
\date{}
\maketitle

\begin{abstract}

A configuration of a graph is an assignment of   one of two states,
on or off, to each vertex of it.    A regular move at a vertex
changes the states of the neighbors of that vertex. A valid move is
a regular move at an on vertex.  The following result is proved in
this note: given any starting configuration $x$ of a tree, if there
is a sequence of regular moves which brings $x$ to another
configuration in which there are $\ell$ on vertices then there must
exist a sequence of valid moves which takes $x$    to a
configuration with at most $\ell +2$ on vertices.  We provide
example to show that the upper bound $\ell +2$ is sharp.  Some
relevant results and conjectures are also reported.
\\

\end{abstract}

\section{Sigma-game}

We consider graphs without multiple edges but may have loops. That
is, for any graph $G$ with {\em vertex set} $V(G)$, its {\em edge
set} $E(G)$ is a subset of ${V(G)\choose 2}\cup V(G)$ and we say
that there is a {\em loop} at a vertex $v$ of $G$ provided $\{v\}\in
E(G)$. Two vertices $u$ and $v$, possibly equal, are {\em adjacent}
in $G$ if $\{u\}\cup \{v\}\in E(G)$. The {\em neighbors} of $v$ in
$G,$ denoted $N_G(v)$, is the set of vertices adjacent to $v$ in
$G.$  For any $v\in V(G)$,  $\chi_v\in \mathbb{F}_2^{V(G)}$ stands
for  the  binary  function for which  $\chi_v(u)=1$ if $u=v$ and
$\chi_v(u)=0$ otherwise.  For any $S\subseteq V(G)$, we set $\chi_S$
to be   $\sum_{v\in S}\chi_v\in \mathbb{F}_2^{V(G)}$.

The so-called {\em sigma-game} on a graph $G$ is a solitaire
combinatorial game widely studied in the literature
\cite{BS,Beal,BL,CLWZ,Clausing,CHKS,DW,EES,Gassner,GLR,GK97,GKW,GK05,Hatzl,Joyner,LWZ,LZ,LL,MN,Meu,Sutner89,Sutner00,Zai}.
Let us call each element of $\mathbb{F}_2^{V(G)}$ a {\em
configuration} of $G$. We can think of a configuration $x$ as an
assignment of one of two states, ¡°on¡± or ¡°off¡±, to the vertices
of $G$ such that $v$ is on if $x(v)=1$ and $v$ is off if $x(v)=0.$
The {\em light number} of a configuration $x$, written as $L(x)$,
refers to the number of vertices which are assigned the on state  by
$x.$ Given any configuration $x$, a {\em move} of the sigma-game is
to pick  a vertex $v$ and toggle the states of all its neighbors
between on and off, namely to transform the configuration $ x$ into
$ x+
 \chi_{ N_G(v)}$. The goal of the sigma-game is to transform a
given configuration by repeated moves to some configuration which is
as good as possible according to certain criterion, say minimizing
the light number of  the configuration.

If we restrict the moves of the sigma-game at on  vertices only,
then we come to the
 {\em lit-only sigma-game}. That is, when $v$ is off, an {\em invalid} move at $v$
 keeps the configuration unchanged; when $v$ is on, a {\em valid} move at $v$ transforms
  the configuration $ x$ into $ x+ \chi_{ N_G(v)}$. In all, a move of the lit-only sigma-game
  at $v$ transforms the configuration $ x$ into
  $ x+x(v) \chi_{ N_G(v)}= x(I+\chi_v^\top \chi_{ N_G(v)})$,
where $I$ is the identity
  matrix.  The introduction of the lit-only restriction makes the sigma-game
    harder to analyze and leads to an even  richer mathematical structure
    \cite{CH06,CH07,Cr,EES,F01,F,F1,GK,GWW,HW,HW1,WW1,WW,WW09,WU}.

 For $x,y\in
\mathbb{F}_2^{V(G)}$, we write   $ x\to _G y$ to mean that
 $x$ can be transformed  to $ y$  by successive  moves in the sigma-game on
 $G$.
 Correspondingly, we
write $ x\xrightarrow{*}_Gy$ if we can go from $x$ to $y$ by
applying a sequence of valid moves in the lit-only  sigma-game on
$G$. Note that we often drop the subscript $G$ from the notation if
it is clear from the context.  When we consider both sigma-game and
lit-only sigma-game, we often call   the moves in the sigma-game
{\em regular
    moves} to distinguish it from the valid/invalid moves in the lit-only
    sigma-game.

The  sigma-game is invertible, namely $ x\to y$ if and only if $
y\to x$, and
 the order in which we execute the moves  is irrelevant. Indeed, $ x\to y$
if and only if $ x-y$ lies in the abelian group generated by $
\chi_{ N_G(v)}, v \in V(G)$. To study the sigma-game is just to
study the action of this abelian group on $\mathbb{F}_2^{V(G)}.$

In general, the lit-only sigma-game may be unilateral, namely $
x\xrightarrow{*}y$ does not imply
 $ y\xrightarrow{*}x$, and the order of moves is significant. In
 other words, the study of lit-only sigma-game is a study of the
 action of a semigroup, which is rarely abelian, on
 $\mathbb{F}_2^{V(G)}.$
 As a trivial example of non-invertibility, consider the graph $G$ with $E(G)=V(G)\cup {V(G)\choose
 2}$ and we can easily find that $1=\chi _{
 V(G)}\xrightarrow{*}
 0$ and that $0$ cannot go anywhere else in the lit-only sigma game.
 Clearly, the existence of loops causes the intricate issue of
 non-invertibility for the lit-only sigma game. On the other hand,
     if $G$ has no loops, then it holds for all $v\in V(G)$
 that
   $(I+\chi_v^\top \chi_{ N_G(v)})^2=I$, implying that the
     lit-only  sigma-game is invertible in this case,
     and $\{ y\in\mathbb{F}_2^{V(G)}\mid x\xrightarrow{*}y\}$
     forms an orbit under the action of the group $H$  generated by
     $\{ I+\chi_v^\top \chi_{ N_G(v)}: v\in V(G)\}$.

     We will elaborate more carefully in Section 2 on the influences of the lit-only restriction to
     the sigma-game  and present proof details in Section 3  to
     confirm
     some of those described influences.

\section{Influences of the lit-only restriction}

\begin{definition}
Let $  x\in \mathbb{F}_2^{V(G)}$.
\begin{itemize}
\item $ML_G( x)=\displaystyle\min_{ x\to _Gy}L( y)$ is called the {\em minimum light number} of
$ x$ for the sigma-game on $G$.
\item $ML^*_G( x)=\displaystyle\min_{ x\xrightarrow{*}_Gy}L( y)$ is called the {\em minimum light
 number} of $ x$ for the lit-only sigma-game on $G$.
\item $ML(G)=\displaystyle\max_{ x\in\mathbb{F}_2^n}ML_G( x)$ is called the {\em minimum light number} of
 the sigma-game on $G$.
\item
$ML^*(G)=\displaystyle\max_{ x\in\mathbb{F}_2^n}ML^*_G( x)$ is
called the {\em minimum light number} of
 the lit-only sigma-game on $G$.
\end{itemize}
\end{definition}

To understand the influences of the lit-only restriction, a basic
question to answer is the following.

\begin{problem}\label{prob1}
Suppose $ x,y\in\mathbb{F}_2^{V(G)}$, $ x\to_G y$ and $L( y)=ML_G(
x)$.
 When can we conclude that $ x\xrightarrow{*}_Gy$? How large can $ML^*_G( x)-ML_G( x)$ be? How
 large can $ML^*(G)-ML(G)$ be?
\end{problem}

It is obvious that \begin{equation}ML_G( x)\le ML^*_G( x), ML(G)\le
ML^*(G). \label{eq1}\end{equation}

We flesh out a bit the above general observation by presenting some
examples, which say that both $ML^*_G(x)-ML_G(x)$ and $ML^*(G)-ML(G)$
can be arbitrarily large and both equalities in display \eqref{eq1} hold for
infinitely many graphs.

\begin{example}\label{exam3}  \cite{GK}
Let $G=K_{m,m,m}$ be the complete tripartite graph,  namely
$V(G)=\{v_{ij}:\ i=1,2,3,j=1,2,\ldots, m\}$ and $E(G)$ consists of
all those pairs $v_{ij}v_{k\ell}$ where $i\not= k.$  Let $x$ be the
configuration of  $G$ such that $x(v_{ij})=0$ if and only if $i=1.$
Then,  we have
$$ML_G( x)=0,~~ML^*_G( x)=2m,~~ML(G)=\left\lfloor\frac{3m}{2}\right\rfloor,~~ML^*(G)=2m.$$
Note that $$ML^*(x)-ML(x)=2m=\frac{2}{3}|V(G)|,
ML^*(G)-ML(G)=\frac{m}{2}=\frac{|V(G)|}{6}.$$
\end{example}

With Example \ref{exam3} in mind, it was conjectured that if $G$ is
a graph of order $n$ with no isolated vertices, then
$ML^*(G)-ML(G)\leq \lceil \frac{n}{6}\rceil $ \cite[Conjecture
4]{GK}. The next result is a counterexample to this conjecture.

\begin{example}\label{exam4}
Let $G=K_{2m}$ be the complete graph without loops on  $2m$
vertices, and $x$ be any element of $\mathbb{F}_2^{V(G)}$ with
$L(x)=m$.    Then we have
$$ML({ x})=ML(G)=0,~~ML^*({ x})=ML^*(G)=m.$$
Hence $$ML^*(G)-ML(G)=ML^*(x)-ML(x)=m=\frac{1}{2}|V(G)|.$$
\end{example}

\begin{example}
Let $G$ be a graph with loops everywhere. As    an easy consequence
of \cite[Theorem 3]{GWW}, we know that  $ML^*_G( x)=ML_G( x)$ is
valid for any configuration $x$ of  $G$ and hence
$ML^*(G)=ML(G)$.\label{thm2}
\end{example}

Besides the above result for  graphs with loops everywhere, there
are results for trees  from  which
 one can  also see  that  the lit-only restriction  does not matter too much.
  The {\em degree} of a vertex $v$ in a
graph $G$ is defined to be the number of edges in $E(G)\setminus
V(G)$ which contains $v$ and we will use the notation $\deg_G(v)$
for it.  A vertex of degree one is said to be a  {\em leaf}.

\begin{example} \cite{WW1,WW}   Let $G$ be any tree, $G'$ be a graph obtained from
$G$ by adding some loops, and $\ell$ the number of leaves of $G$. If
$\ell\geq 2,$ then $ML(G')\leq \lfloor \ell /2\rfloor$   and
$ML^*(G)\leq \lceil \ell/2 \rceil$. Both equalities can be attained.
\label{exam5}
\end{example}

Note that in Example  \ref{exam5} we do not
 directly compare the difference between the minimum light numbers
 of the sigma-game and the lit-only sigma-game, which is an object
 of interest posed in both   \cite[\S
1.3]{GWW} and  \cite[Question 3]{GK}. Let us make the following two
conjectures regarding it here, the second of which being motivated
by Example \ref{exam4}.

\begin{conjecture}Let $G$ be obtained from a tree by adding some
loops. Then $ML^*(G)-ML(G)\in \{0, 1\}.$  \label{conj1}
\end{conjecture}

\begin{conjecture}\label{conj2} It holds for any graph $G$ that
$ML^*(G)-ML(G)\le\frac{1}{2}|V(G)|$.
\end{conjecture}

Suppose that $y,z$ are two configurations of a graph $G$ such that
$ML^*(G)=ML_G^*(z)$ and $ML(G)=ML_G(y)$. Since  $ML^*_G(
z)-ML_G(y)\leq ML^*_G( z)-ML_G(z)$, we infer that
\begin{equation}0\leq ML^*(G)-ML(G)\leq \max _{x\in
\mathbb{F}_2^{V(G)}}(ML^*_G( x)-ML_G( x)). \label{eq2}\end{equation}
This suggests  that, instead of tackling  Conjecture \ref{conj1}
 and/or Conjecture \ref{conj2}
 directly, we may first try to find an upper bound
for $ML^*_G( x)-ML_G( x)$   for any configuration $x$  of some
special graph $G$.

Let $G$ be a graph obtained from a tree by attaching some loops and
$x$ a configuration of $G$. According to  Example \ref{thm2},
$ML^*_G( x)-ML_G( x)$ will take value $0$ if $G$ has loops
everywhere. We proceed to give two more examples to show that it is
possible that $ML^*_G( x)-ML_G( x)$ takes value $1$ or $2$ as well.
To demonstrate a configuration, we will draw the underlying graph
and use a bullet to indicate an on vertex and use a circle for an
off vertex.

\begin{figure}[ht]
\setlength{\unitlength}{0.5cm}
\begin{center}\begin{picture}(16,3)
\put(-2,0.5){\makebox(0,0){$x:$}}
\multiput(2.2,0.5)(2,0){6}{\line(1,0){1.6}}
\multiput(4,0.5)(2,0){5}{\circle{0.4}} \put(8,0.7){\line(0,1){1.6}}
\put(2,0.5){\circle*{0.4}} \put(8,2.5){\circle{0.4}}
\put(14,0.5){\circle*{0.4}} \put(2,0){\makebox(0,0){$v_1$}}
\put(4,0){\makebox(0,0){$v_2$}} \put(6,0){\makebox(0,0){$v_3$}}
\put(8,0){\makebox(0,0){$v_4$}} \put(8,3){\makebox(0,0){$v_5$}}
\put(10,0){\makebox(0,0){$v_6$}} \put(12,0){\makebox(0,0){$v_7$}}
\put(14,0){\makebox(0,0){$v_8$}}
\end{picture}\end{center}
\caption{$ML^*(x)-ML(x)=1$}\label{fig1}
\end{figure}
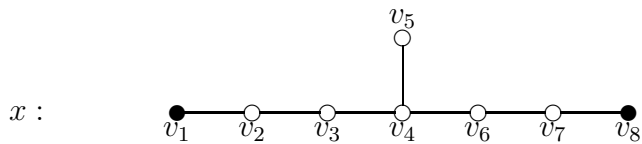

\begin{example}  Let $x$ be the configuration depicted in Fig.
\ref{fig1}. A computer search demonstrates that $ML_G(x)=1$ and
$ML^*_G(x)=2$.
\end{example}

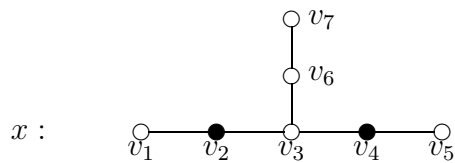
\begin{figure}[ht]
\setlength{\unitlength}{0.5cm}
\begin{center}\begin{picture}(8,4)
\put(-3,0.5){\makebox(0,0){$x:$}}
\multiput(0,0.5)(2,0){5}{\circle{0.4}}
\multiput(4,2)(0,1.5){2}{\circle{0.4}} \put(2,0.5){\circle*{0.4}}
\put(6,0.5){\circle*{0.4}}
\multiput(0.2,0.5)(2,0){4}{\line(1,0){1.6}}
\multiput(4,0.7)(0,1.5){2}{\line(0,1){1.1}}
\put(0,0){\makebox(0,0){$v_1$}} \put(2,0){\makebox(0,0){$v_2$}}
\put(4,0){\makebox(0,0){$v_3$}} \put(6,0){\makebox(0,0){$v_4$}}
\put(8,0){\makebox(0,0){$v_5$}} \put(4.8,2){\makebox(0,0){$v_6$}}
\put(4.8,3.5){\makebox(0,0){$v_7$}}
\end{picture}\end{center}
\caption{$ML^*(x)-ML(x)=2$}\label{fig2}
\end{figure}

\begin{example}\label{ex3}  Let $x$ be the  configuration depicted in Fig. \ref{fig2}. Then,
$ML(x)=0, ML^*(x)=2$ \cite[p. 299]{WW}.
\end{example}

\begin{figure}[ht]
\begin{center}\begin{picture}(5,2)
\put(2,1){\circle*{0.15}} \put(2,1){\line(2,1){1}}
\put(2,1){\line(2,-1){1}} \put(3,0.5){\line(1,0){2}}
\put(3,1.5){\line(1,0){2}} \put(3,0.5){\circle*{0.15}}
\put(3,1.5){\circle*{0.15}} \put(4,0.5){\circle*{0.15}}
\put(4,1.5){\circle*{0.15}} \put(5,0.5){\circle*{0.15}}
\put(5,1.5){\circle*{0.15}} \put(1.5,1){\circle{2}}
\put(0,1){\makebox(0,0){$G:$}} \put(1.2,1){\makebox(0,0){$G_1$}}
\put(1.7,1){\makebox(0,0){$\cdots$}}
\put(1.7,1.2){\makebox(0,0){$\mathinner{\mkern1mu\raise7pt\hbox{.}\mkern2mu\raise4pt\hbox{.}\mkern2mu\raise1pt\hbox{.}\mkern1mu}$}}
\put(1.7,0.8){\makebox(0,0){$\mathinner{\mkern1mu\raise1pt\hbox{.}\mkern2mu\raise4pt\hbox{.}\mkern2mu\raise7pt\hbox{.}\mkern1mu}$}}
\put(2,0.7){\makebox(0,0){$v$}} \put(3,1.8){\makebox(0,0){$v_{11}$}}
\put(3,0.2){\makebox(0,0){$v_{21}$}}
\put(4,1.8){\makebox(0,0){$\cdots$}}
\put(4,0.2){\makebox(0,0){$\cdots$}}
\put(5,1.8){\makebox(0,0){$v_{1n_1}$}}
\put(5,0.2){\makebox(0,0){$v_{2n_2}$}}
\end{picture}\end{center}
\caption{A graph with two pendant paths}\label{fig3}
\end{figure}
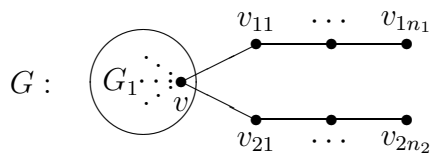

The main idea of   our current approach towards  tackling Problem
\ref{prob1} is reflected in the result below.

\begin{theorem}\label{thm8} Let $G_1$  be a connected   graph with  $v\in V(G_1)$. Let $G$
be the graph obtained from $G_1$ by adding a set of new vertices
$S$$=$$\{$$v_{11},$$ \ldots,$$ v_{1n_1},$$v_{21},$$\ldots,$$
v_{2n_2}\}$,   $n_1,$$n_2\ge 1,$ and adding a set of new edges
$\{$$vv_{11},$$
vv_{21},$$v_{11}$$v_{12},$$\ldots,$$v_{1,n_1-1}v_{1n_1},$$v_{21}v_{22},\ldots,v_{2,n_2-1}v_{2n_2}\}$
as well as some loops inside $S$; see Fig. \ref{fig3}. For any
$x\in\mathbb{F}_2^{V(G)}$, we will have  $ML^*_G( x)-ML_G( x)\le 2$
provided either
 $\max (n_1,n_2)\ge 2$ or $G$
has a loop at either $v_{11}$ or $v_{21}$, or
$x(v_{11})\not=x(v_{21})$.
\end{theorem}

In light of  \eqref{eq2}, what comes next may be viewed as a partial
support to Conjecture \ref{conj1}.

\begin{theorem}\label{thm9}
If $G$ is obtained from a tree by adding some loops, then it holds
$ML^*_G( x)-ML_G( x)\le 2$  for any  $ x\in\mathbb{F}_2^{V(G)}$ and
the upper bound $2$ is sharp.
\end{theorem}

Let $G$ be obtained from a tree by adding some loops and let $x\in
\mathbb{F}_2^{V(G)}.$ We point out that Mu Li designed a polynomial
algorithm to derive $ML_G( x)$  \cite{LI} and some relevant
 complexity results for determining $ML_G(x)$ for tree-like graphs $G$ can be found in \cite{Gassner,Meu}. For our present purposes, it is pertinent to
ask what is the complexity of calculating $ML^*_G( x)$.

For any graph $G$ and $S\subseteq V(G)$, the subgraph of $G$ induced
by $S$ is denoted $G[S]$.  We say that $v$ is a {\em branch vertex}
of $G$ if $\deg_G(v)\geq 3.$  Going through the procedure of
establishing Theorems \ref{thm8} and \ref{thm9}, it will be not hard
to realize the following result about a graph obtained by `planting'
a tree with at least 3 branch vertices on a connected graph.  Note
that the graph $G$  treated in Theorem \ref{thm8} can be viewed as a
graph obtained by `planting' a path on the connected graph  $G_1.$

\begin{theorem}\label{thm10}
Let $G$ be a  graph with $V(G)=V_1\cup V_2$ and $V_1\cap V_2=\{v\}$.
  Suppose that $G[V_1]$ is a tree with some loops attached and contains at least three branched vertices of itself
   and  $G[V_2]$  is connected.
    If there is no edge in $G$  which intersects both $V_1\setminus
    \{v\}$  and $V_2\setminus
    \{v\}$, then it holds  $ML^*_G( x)-ML_G( x)\le 2$  for any  $
    x\in\mathbb{F}_2^{V(G)}$.
\end{theorem}

Employing similar idea as in the proof of preceding theorems
 but with more technical details, we can carry out the
proof of the following two results, which will be reported in a
subsequent paper \cite{WW09}.

\begin{theorem}\label{thm4}
If $G$ is unicyclic, then it holds $ML^*_G( x)-ML_G( x)\le 3$ for
any $ x\in\mathbb{F}_2^{V(G)}$. This bound is tight.
\end{theorem}

\begin{theorem}\label{thm5}
If $G$ is obtained from a grid graph by adding some loops, then
$ML^*_G( x)-ML_G( x)\le 3$ for any configuration $x$.
\end{theorem}

 Theorems  \ref{thm9},
\ref{thm4} and \ref{thm5} stimulate us to set forth the following
conjecture.

\begin{conjecture} For any graph $G$ and any $
x\in\mathbb{F}_2^{V(G)}$, it holds  $ML^*_G( x)-ML_G(
x)\le\max\limits_{v\in V(G)}\deg_G(v)$.
\end{conjecture}

\section{Proofs}

This concluding section is devoted to  proofs of  Theorems
\ref{thm8},    \ref{thm9} and \ref{thm10}.

\begin{lemma}    Let $G$ be a connected graph and $0\not=x \in \mathbb{F}_2^{V(G)}$.
Suppose $a$ and $b$ are two vertices of $G$ satisfying $N_G(a)\not=
N_G(b)$.    Then, there is $y\in \mathbb{F}_2^{V(G)}$ such that
 $y(a)\not= y(b)$ and $
x\xrightarrow{*}_Gy$.\label{lem14}\end{lemma}
\begin{proof} Without loss of generality, assume that there is $c\in N(a)\setminus
N(b)$.
 If $x(c)=1$, then a  valid move at $c$, if necessary,  will bring us to a
configuration where the states of $a$ and $b$ are different.
Consequently, it remains  to show that there is $z\in
\mathbb{F}_2^{V(G)}$ such that
 $z(c)=1$ and $
x\xrightarrow{*}_Gz$.

Since $0\not= x,$ we can take $d\in V(G)$ such that $x(d)=1.$ Choose
a shortest path connecting $d$ and  $c$ in $G,$ say $w_0w_1\cdots
w_t$, where $w_0=d$  and $w_t=c.$ If $x(c)=1$, it suffices to put
$x=z$. Otherwise, let $q$ be the largest  integer less than $t$ such
that $x(w_q)=1$. A sequence of valid moves at $w_q,w_{q+1},\ldots,
w_{t-1}$ transforms $x$ to a configuration $z$ satisfying $z(c)=1$,
finishing the proof.
\end{proof}

\begin{lemma}\label{lem15} Let $G$ be a  graph,  $a,b\in V(G)$,  $ ab\notin E(G),$   $c\in N_G(a)\cap
 N_G(b)$. Let
 $   S\subseteq  V(G)\setminus (
N_G(a) \cup  N_G(b))$ such that
  $G[S\cup \{c\}]$ is
connected.  Assume that  $x$ and $y$ are two configurations of $G$
such that $x\to_G y$. If $x(a)\not= x(b)$, then there exists    $R
\subseteq V(G)\setminus (S\cup \{c\})$ such that
$x\xrightarrow{*}_Gy+\sum_{v\in R}\chi_{N(v)}$.
\end{lemma}

\begin{proof} Our goal   is to show that for any finite set $T\subseteq
S\cup \{c\}$, there exists $R' \subseteq V(G)\setminus (S\cup
\{c\})$ such that $x\xrightarrow{*}_G x+\sum_{v\in T\cup
R'}\chi_{N(v)}$. This can be accomplished by inductively  appealing
to the following claim: let $d\in S\cup \{c\}$ be the vertex whose
distance to $c$ in the graph $G[S\cup \{c\}]$ is the largest, say
$D$, and let $S'$ be the set of those vertices in  $S\cup \{c\}$
which have a distance less then $D$  to $c$ in $G[S\cup \{c\}]$,
then we can find $U \subseteq S'\cup \{a,b\}$ such that
$x\xrightarrow{*}_G w$ where $w=x+\chi_{N(d)}+\sum_{v\in
U}\chi_{N(v)}$ and $w(a)\not= w(b)$. To establish this claim, we
choose   a shortest path in $G[S\cup \{c\}]$ which connects $c$ and
$d$, say $v_1v_2\cdots v_{D+1}$ where $v_1=c$ and $v_{D+1}=d$.
  Let us refer to the only on vertex among $\{a,b\}$ in the configuration $x$ as $v_0$
  and let $t$ be
the largest integer no greater than $D+1$ such that $x(v_t)=1$. What
is left to do is to distinguish two cases.

\paragraph{\sc Case 1:} Either $t>0$ or there is no loop at $v_0$ in $G.$  It is
easy to check that the valid moves at $v_t,v_{t+1},\ldots,v_{D+1}$
in that order  transforms $x$ to the required $w$.

\paragraph{\sc Case 2:}  There is a loop at $v_0$ in $G$ and $t=0$. The sequence  of valid  moves at
$v_0,v_1,\ldots,v_{D+1},v_0$ successively is what we want.
\end{proof}

We now arrive at  the  key ingredient in our proofs of Theorems
\ref{thm8}, \ref{thm9}, \ref{thm10}, \ref{thm4} and \ref{thm5}.

\begin{lemma}\label{lem1516} Let $G$ be a  graph,  $a,b\in V(G)$,  $ ab\notin E(G),$   $c\in N_G(a)\cap
 N_G(b)$ and $N_G(a)\not= N_G(b)$. Let
 $   S\subseteq  V(G)\setminus (
N_G(a) \cup  N_G(b))$ such that
  $G[S\cup \{c\}]$ is
connected.  Assume that  $x$ and $y$ are two configurations of $G$
such that $x\to_G y$. If $x\not= 0$, then there exists    $R
\subseteq V(G)\setminus (S\cup \{c\})$ such that
$x\xrightarrow{*}_Gy+\sum_{v\in R}\chi_{N(v)}$.
\end{lemma}
\begin{proof}Lemma \ref{lem14} in conjunction with Lemma \ref{lem15}  gives this result. \end{proof}

\begin{lemma}\label{lem16}
If $G$ is obtained from a path $v_1v_2\cdots v_n$ by adding some
loops, then any configuration $x$ of $G$ can be transformed to a
configuration with light number at most one by a series of valid
moves inside $\{v_2,\ldots,v_n\}$.
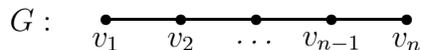
\begin{center}\begin{picture}(5,0.8)
\put(1,0.5){\line(1,0){4}} \put(1,0.5){\circle*{0.15}}
\put(2,0.5){\circle*{0.15}} \put(3,0.5){\circle*{0.15}}
\put(4,0.5){\circle*{0.15}} \put(5,0.5){\circle*{0.15}}
\put(0,0.5){\makebox(0,0){$G:$}} \put(1,0.2){\makebox(0,0){$v_1$}}
\put(2,0.2){\makebox(0,0){$v_2$}} \put(3,0.2){\makebox(0,0){\ldots}}
\put(4,0.2){\makebox(0,0){$v_{n-1}$}}
\put(5,0.2){\makebox(0,0){$v_n$}}
\end{picture}\end{center}
\end{lemma}

\begin{proof} Let $t_y=\infty $ if $y=0$ and $t_y=\min \{t: y(v_t)=1\}$ for any $y\in
\mathbb{F}_2^{V(G)}\setminus \{0\}$. Let $\mathcal {C}$ be the set
of configurations for which we can reach from $x$ by applying a
series of valid moves inside $\{v_2,\ldots,v_n\}.$  Choose  a
configuration $y$ from $\mathcal {C}$ whose  $t_y$ is as large as
possible. It suffices to deduce that $L(y)\leq 1.$ Assuming
otherwise that $L(y)>1$, then there is $t>t_y$ such that $y(v_t)=1$
and $y(v_q)=0$ for any $t_y<q<t.$ Now a series of valid moves at
$v_t,v_{t-1},\ldots, v_{t_y+1}$ transforms $y$ into another member
$y'$ of $\mathcal {C}$ with $t_{y'}>t_y$, yielding a contradiction.
\end{proof}

Our proof of Theorem \ref{thm8} rests on Lemmas  \ref{lem15},
\ref{lem1516} and \ref{lem16}.

\begin{proof}[Proof of Theorem \ref{thm8}]  If $x=0$, then
$ML^*_G(x)=ML_G(x)=0$ and hence we are done. Now consider $x\not=
0.$ Choose $y$ such that $x\to_G y$ and $L(y)=ML_G(x)$. Taking
$a=v_{11}$, $b=v_{21}$ and $c=v$, we deduce from  Lemmas
\ref{lem1516} (Lemma   \ref{lem15})  that there exists $z\in
\mathbb{F}_2^{V(G)}$ such that $x\xrightarrow{*}_Gz=y+\sum_{v\in
R}\chi_{N(v)}$ where $R$ is a subset of $S$. This means that
$z(u)=y(u)$ for all $u\in V(G)\setminus (S\cup \{v\})$. Therefore,
the result will follow if we can show that there exists  a series of
valid moves inside $S$ which transforms $z$ to $z'$ where $z'$ is a
configuration with at most two on vertices among $S\cup \{v\}.$ This
last step is completed by making use of  Lemma \ref{lem16} on
$G[\{v,v_{11},\ldots,v_{1n_1}\}]$  and
$G[\{v,v_{21},\ldots,v_{2n_2}\}]$, respectively.
\end{proof}

\begin{remark}  In much the same vein as   the above proof of Theorem \ref{thm8},  we can
prove the following extra claim where
all  undefined  parameters  are as described in Theorem \ref{thm8}:
Putting $x_1$ to be the restriction of $x$ on $V(G_1)$, we have
$$ML^*_G(x)\leq ML_{G_1}(x_1)+2$$  provided either
 $\max (n_1,n_2)\ge 2$ or $G$
has a loop at either $v_{11}$ or $v_{21}$, or
$x(v_{11})\not=x(v_{21})$.
\end{remark}

  We are going to establish Theorem
\ref{thm9} in the sequel.   To do that, we need to develop  some
special facts on trees.

 Let $G$ be any graph and $v\in V(G)$. We call $v$
  an {\em end
vertex} of $G$ if $\deg_G(v)\leq 1$.   A subset of
      $V(G)$ is {\em good} for
$G$ if it does not contain any branch vertex of $G$.
   The vertex $v$ of
$G$  is  {\em appropriate} \cite{BFS,Nylen} provided  there  are at
least two  (connected) components  of  $G[V(G)\setminus \{v\}]$
which are good for $G$.  Here is a very intuitive result, which may
be traced back to   a paper of Nylen  \cite{Nylen}.

\begin{lemma}   \cite[Lemma 3.1]{Nylen}  Let $H$ be obtained from a tree by attaching some loops.
Suppose $H$ has at least two vertices. Then there exists a path in
$H$
 such that the path contains
two different leaves of $H$ and   contains  at most one branch
 vertex of  $H$. \label{lem17}\end{lemma}

\begin{proof}
The result is trivial if $H$ does not contain any branch vertex. For
the remaining case, consider the  graph $G=H[S]$ where $S$ is the
inclusion-wise smallest  subset of $V(H)$ which contains all branch
vertices of $H$  and makes $H[S]$ connected. Note that   $G$
definitely has some end vertex $v$.
 It is easy to see that all  components of $H[V(H)\setminus
 \{v\}]$ other than that which contains $S$  are good for $H$ and there
 are at least two such components.
This says that
 $v$ is an  appropriate vertex of $H$ and  the required path can be obtained by
 combining  $v$ with  any two  components of
$H[V(H)\setminus \{v\}]$ which  are good for $H$.\end{proof}

\begin{remark}\label{re21} Keep the  assumption on $H$ as in Lemma \ref{lem17} and let $u$ be a leaf of $H$.
The above proof of Lemma \ref{lem17} could be extended a bit more to
yield an inductive proof of  the following still `obvious' fact: if
$H$ has at least three leaves, then the path asserted in Lemma
\ref{lem17}  can be further required to avoid $u$.
\end{remark}

For any two integers $n,k\geq 1,$  let $P_{n,k}$ be the graph with
vertex set $\{v_1,v_2,\ldots,v_{n},w_1,\ldots,w_k\}$ and edge set
$\{ v_1v_2,$ $v_2v_3,$ $\ldots,$ $v_{n-1}v_n,$ $v_{n}w_1,$ $\ldots,$
$v_{n}w_k\}$; see Fig. \ref{fig4}.   We refer to any graph obtained
from  $P_{n,k}$ by attaching some loops as  a {\em rake with $k$
teeth} $w_1,\ldots,w_k$ and an {\em $n$-handle} $v_1,\ldots,v_n$. We
call the   vertex $v_1$ the {\em top} of the rake and the other
vertices the {\em common vertices}. When $k=1,$ $P_{n,k} $ is just a
path of length $n$ one of whose end vertices is specified as the
top. In preparation for our proof of Theorem \ref{thm9}, we prove
the following simple fact, whose role will be similar to that of
Lemma \ref{lem16} in proving Theorem \ref{thm8}.

\begin{figure}[ht]
\begin{center}\begin{picture}(6,1)
\put(0,0.5){\line(1,0){5}} \put(4,0.5){\line(2,1){1}}
\put(4,0.5){\line(2,-1){1}} \multiput(0,0.5)(1,0){6}{\circle*{0.15}}
\multiput(5,0)(0,1){2}{\circle*{0.15}}
\put(0,0.2){\makebox(0,0){$v_1$}} \put(1,0.2){\makebox(0,0){$v_2$}}
\put(2,0.2){\makebox(0,0){$\ldots$}}
\put(3,0.2){\makebox(0,0){$v_{n-1}$}}
\put(4,0.2){\makebox(0,0){$v_n$}} \put(5.5,1){\makebox(0,0){$w_1$}}
\put(5.4,0.6){\makebox(0,0){$\vdots$}}
\put(5.5,0){\makebox(0,0){$w_k$}}
\end{picture}\end{center}
\caption{A rake with $k$ teeth and an $n$-handle}\label{fig4}
\end{figure}
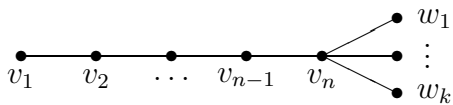

\begin{lemma}\label{lem20} Let $G$ be a rake with $k$ teeth and an $n$-handle;
 see Fig. \ref{fig4}.
     For any configurations $x,y$ of
$G$,  if there is a sequence  of regular   moves inside the common
vertices of $G$ which brings $x$
 to  $y$, then there exists $z\in \mathbb{F}_2^{V(G)}$ satisfying  $L(z)\leq L(y)+1$ and  a series of valid moves inside
the common vertices which transforms $x$ to $z$.
\end{lemma}
\begin{proof} For any set $S\subseteq V(G)$, denote by $d(S)$ the minimum one
among all distances in $G$ between the top $v_1$ and elements of
$S.$ For any $u\in V(G)\setminus \{v_1\}$, put $\overleftarrow{u}$
to be the set of vertices other than $u$ in the shortest path
connecting $u$ and $v_1$, $f(u)$ be the unique vertex adjacent to
$u$ falling in $\overleftarrow{u}$, and set $\overrightarrow{u}$ to
be $V(G)\setminus \overleftarrow{u}$.

Suppose that  by valid moves inside $V(G)\setminus \{ v_1\}$ we get
to $x'=y+\sum_{v\in S_{x'}}\chi_{N_G(v)}$ from $x$ for some set
 $S_{x'}\subseteq V(G)\setminus \{ v_1\}$.  There are three cases
to consider.   If $S_{x'}$ is empty, the proof is completed by
setting $z=x'$. Otherwise, we pick a vertex from $S_{x'}$, say $u$,
whose distance to $v_1$ equals $d(S_{x'})$. Note that $x'$ coincides
with $y$ when restricted on  $\overleftarrow{f(u)}$. If  $x'$
vanishes on $\overrightarrow{u}$, then  the only possible on vertex
for $x'$ inside $V(G)\setminus \overleftarrow{f(u)}$ is $f(u)$  and
this demonstrates that $L(x')\leq   L(y)+1$ and so the required $z$
can still be taken as $x'$. For the moment, it is sufficient to
consider the third case that vertices in $\overrightarrow{u}$ are
not all off in the assignment $x'$. We can thus find
 $u'\in \overrightarrow{u}$ such that $x'(u')=1$ and $x'(u_1)=x'(u_2)=\cdots =x'(u_t)=0$
 where $u_1u_2\cdots u_tu'$ is the shortest path connecting $u$ and
 $u'$  (Note that $t=0$ when $u=u'$.).  A sequence  of valid moves
 along $u',u_t,\ldots,u_1$ turns $x'$ into $x''$ for which we have the following:  \begin{enumerate}
\item[{\textrm (i)}] $ x\xrightarrow{*}_Gx''$;
\item[{\textrm (ii)}] There exists $S_{x''}\subseteq     V(G)\setminus \{ v_1\}$
such that $y=x''+\sum_{v\in   S_{x''}}\chi_{N(v)}$ and either
$d(S_{x''})>d(S_{x'})$ or $d(S_{x''})=d(S_{x'})=d$ but  $\{v\in
S_{x''}:\  d(\{v\})=d\}$ is a proper subset of  $\{v\in S_{x'}:\
d(\{v\})=d\}$. Note that the former case happens  if
$S_{x'}\setminus \{w_1,\ldots,w_k\}\not= \emptyset.$
 \end{enumerate}
At this stage it is not difficult to see that we can apply the above
procedure  repeatedly  and finally terminate at one of the first two
cases. This ends  the proof.
\end{proof}

\begin{remark}\label{re24}  Mimicking the above proof of Lemma \ref{lem20},
 it is easy to show that $ML^*_G(x)\leq ML_G(x)+1$ for any rake  $G$  and   $x\in \mathbb{F}_2^{V(G)}$.
\end{remark}

Having derived Lemmas  \ref{lem1516},  \ref{lem17} and \ref{lem20},
we are ready to  give a proof of Theorem \ref{thm9}.

\begin{proof}[Proof of Theorem \ref{thm9}]

The tightness of the bound follows from
 Example \ref{ex3} and hence we just need to establish that  bound.

By virtue of Theorem \ref{thm8},  we could and will make the
following assumption from now on: for any appropriate vertex $v$ of
$G$, every  component of $G[V(G)\setminus \{v\}]$ that is good for
$G$ contains exactly one vertex and this vertex has no  loop
attached and all these  good components (for the same $v$) are in
the same states in the assignment $x.$  We further  rule out  the
trivial case that $x=0.$ To complete the proof, let us distinguish
two cases.

If there is at most one branch vertex in $G,$ we can infer that $G$
must be a rake and therefore Remark \ref{re24} yields the result.
(Indeed, under the current assumption, we can prove directly a
stronger result that
  $ML^*_G(x)\leq 1$.)

We continue to dwell on the case that $G$  contains at least two
branch vertices. For each appropriate vertex $v$ of $G$, we assume
that $G[V(G)\setminus \{v\}]$ has  $k_v$  components  which are good
for $G$, each of which should be a singleton set by our assumption.
Now, for each appropriate vertex $v$ of $G$,  delete $k_v-1$ good
components corresponding to $v$ as well as their incident edges and
call the resulting graph $H.$   By Lemma \ref{lem17} applied to $H,$
we conclude that there is a path $P$ in $H$ which contains two
different leaves of $H$ and at most one branch vertex of $H.$ We
consider two subcases.

\paragraph{\sc Subcase 1:}  $P=u_1u_2\cdots u_p$ and $P$ does not
contain any branch vertex of $H.$  Since $G$ has at least two branch
vertices, the only possibility is that $u_2$ and $u_{p-1}$ are the
two branch vertices of $G.$    Let $H'$ be the rake with top $u_1$
obtained from $G$ by removing  $S=N_G(u_{2})\cap (V(G)\setminus
V(H))$. Note that any regular move inside $S$ can only affect the
state of $u_2.$  This says that there is a set of regular moves
inside $V(H')$ which takes us from $x$ to a configuration $y$
fulfilling $L(y)\leq ML_G(x)+1.$   Thanks to   Lemma \ref{lem20}, we
can now assert that there is $z$ with $L(z)\leq ML_G(x)+2$ and
$x\xrightarrow{*}_Gz$, implying  $ML^*_G(x)\leq ML_G(x)+2$, as
claimed.

\paragraph{\sc Subcase 2:}  $P=u_1u_2\cdots u_puv_q\cdots v_1$, $p\geq q\geq 1,$ and $u$ is a branch vertex of $H.$
Denote by $U$ the  component of  $G[V(G)\setminus \{u\}]$ containing
$u_p$ and by $V$ the  component  of  $G[V(G)\setminus \{u\}]$
containing $v_q$.  Assume that
 $ x\to _G y$ and $L(y)=ML_G(x)$.
   Take $a=u_p,b=v_q$,
  $c=u$ and $S=V(G)\setminus (\{u\}\cup U\cup V)$.
  Due to our construction of  $H,$  we see   that $p>1.$  This then
  allows us to utilize   Lemma  \ref{lem1516} to get  that  there exists  $R\subseteq U\cup V $ such that
$x\xrightarrow{*}_Gy+\sum_{v\in R}\chi_{N(v)}$.  Observe that both
$G[U\cup\{u\}]$ and $G[V\cup\{u\}]$ are rakes with top $u$.
Henceforth, an application of Lemma \ref{lem20} to $G[U\cup\{u\}]$
 and $G[V\cup\{u\}]$ completes the proof.
\end{proof}

\begin{proof}[Proof of Theorem  \ref{thm10}] This can be done analogous to the previous proof of Theorem \ref{thm9},
with   Remark  \ref{re21} in place of Lemma \ref{lem17}. We omit the
details.
\end{proof}

\newcommand{\journalname}[1]{\textrm{#1}}
\newcommand{\booktitle}[1]{\textrm{#1}}
\newcommand{\papertitle}[1]{\textit{#1}}

\end{document}